\documentclass{article}
\usepackage{amssymb,amsmath,multirow}
\usepackage[dvips]{graphicx}
\newtheorem{thm}{Theorem}
\newtheorem{lem}{Lemma}
\newtheorem{cor}{Corollary}
\newtheorem{rem}{Remark}
\newtheorem{exam}{Example}

\title{A Revisit to Quadratic Programming with One Inequality
Quadratic Constraint via Matrix Pencil \thanks{This research was
supported by Taiwan NSC 98-2115-M-006-010-MY2, by National Natural
Science Foundation of China under Grants 11001006 and
91130019/A011702, and by the fund of State Key Laboratory of
Software Development Environment under grant SKLSDE-2013ZX-13.}}

\author{Yong  Hsia\footnotemark[1]\and  Gang-Xuan Lin\footnotemark[2] \and Ruey-Lin Sheu\footnotemark[2]}

\begin{document}
\maketitle
\renewcommand{\thefootnote}{\fnsymbol{footnote}}

\footnotetext[1]{LMIB of the Ministry of Education; School of
Mathematics and System Sciences, Beihang University, Beijing,
100191, P. R. China ({\tt dearyxia@gmail.com}).}

\footnotetext[2]{Department of Mathematics, National Cheng Kung University, Taiwan
 ({\tt rsheu@mail.ncku.edu.tw}).}

\begin{abstract}
The quadratic programming over one inequality quadratic constraint
(QP1QC) is a very special case of quadratically constrained
quadratic programming (QCQP) and attracted much attention since
early 1990's. It is now understood that, under the primal Slater
condition, (QP1QC) has a tight SDP relaxation ${\rm (P_{\rm SDP})}$.
The optimal solution to (QP1QC), if exists, can be obtained by a
matrix rank one decomposition of the optimal matrix $X^*$ to ${\rm
(P_{\rm SDP})}$. In this paper, we pay a revisit to (QP1QC) by
analyzing the associated matrix pencil of two symmetric real
matrices $A$ and $B$, the former matrix of which defines the
quadratic term of the objective function whereas the latter for the
constraint. We focus on the ``undesired'' (QP1QC) problems which are
often ignored in typical literature: either there exists no Slater
point, or (QP1QC) is unbounded below, or (QP1QC) is bounded below
but unattainable. Our analysis is conducted with the help of the
matrix pencil, not only for checking whether the undesired cases do
happen, but also for an alternative way otherwise to compute the
optimal solution in comparison with the usual
SDP/rank-one-decomposition procedure.

\end{abstract}

{\bf Keywords}:  Quadratically constrained quadratic program, matrix
pencil, hidden convexity, Slater condition, unattainable SDP,
simultaneously diagonalizable with congruence.

\pagestyle{myheadings}
\thispagestyle{plain}

\section{Introduction}
Quadratically constrained quadratic programming (QCQP) is a
classical nonlinear optimization problem which minimizes a quadratic
function subject to a finite number of quadratic constraints. A
general (QCQP) problem can not be solved exactly and is known to be
NP-hard. (QCQP) problems with a single quadratic constraint,
however, is polynomially solvable. It carries the following format
\begin{eqnarray}\label{eq0-1}
\begin{array}{rcl}
{\rm (QP1QC)}:~~ \inf &&F(x)= x^TAx-2f^Tx \cr    %
{\rm s.t.} && G(x)=x^TBx-2g^Tx\le \mu,
\end{array}
\end{eqnarray}
where $A, B$ are two $n\times n$ real symmetric matrices, $\mu$ is a
real number and $f, g$ are two $n\times 1$ vectors.

The problem (QP1QC) arises from many optimization algorithms, most
importantly, the trust region methods. See, e.g., \cite{ga,ms} in
which $B\succ0$ and $g=0$. Extensions to an indefinite $B$ but still
with $g=0$ are considered in \cite{Ben, fw, Martinez, sw, xfgsz}.
Direct applications of (QP1QC) having a nonhomogeneous quadratic
constraint can be found in solving an inverse problem via
regularization \cite{inverse, Golub} and in minimizing the double
well potential function \cite{flsx}. Solution methods for the
general (QP1QC) can be found in \cite{flsx2010,More,pt}.

As we can see from the above short list of review, it was not
immediately clear until rather recent that the problem (QP1QC)
indeed belongs to the class P. Two types of approaches have been
most popularly adopted: the (primal) semidefinite programming (SDP)
relaxation and the Lagrange dual method. Both require the (QP1QC)
problem to satisfy either the primal Slater condition (e.g.,
\cite{flsx2010, More, xfgsz, yz}):
\\[+5pt]
\textbf{Primal Slater Condition}:
There exists an $x_0$ such that $G(x_0)=x_0^TBx_0-2g^Tx_0< \mu$;\ \\[+7pt]
and/or the dual Slater condition (e.g., \cite{Ben, flsx2010, More,
xfgsz, yz}):
\\[+5pt]
\textbf{Dual Slater Condition}:
There exists a $\sigma'\ge 0$ such that $A+\sigma' B\succ0.$
\\[+5pt]
Nowadays, with all the efforts and results in the literature,
(QP1QC) is no longer a difficult problem and its structure can be
understood easily with a standard SDP method, the conic duality, the
S-lemma, and a rank-one decomposition procedure.

Under the primal Slater condition, the Lagrangian function of
(QP1QC) is
\begin{equation}\label{Lagrange}
d(\sigma)=\inf_{x\in{\mathbb R}^n} L(x,\sigma):=x^T(A+\sigma
B)x-2(f+\sigma g)^Tx -\mu\sigma,~~\sigma\ge 0,
\end{equation}
with which the dual problem of (QP1QC) can be formulated as
\[
  {\rm(D)}~~~~ \sup_{\sigma\ge 0}   d(\sigma).
 \]
Using Shor's relaxation scheme \cite{Shor}, the above Lagrange dual
problem (D) has a semidefinite programming reformulation:
\begin{eqnarray}\label{Dsdp}
\begin{array}{rcl}
{\rm (D_{\rm SDP})}: ~~\sup && s \cr    %
 {\rm s.t.} && \left[\begin{array}{cc}A+\sigma B& -f-\sigma g\\-f^T-\sigma g^T&-\mu\sigma-s\end{array} \right]
   \succeq 0,\\
 &&\sigma\geq 0.
\end{array}
\end{eqnarray}
The conic dual of ${\rm (D_{\rm SDP})}$ turns out to be the SDP
relaxation of (QP1QC):
\begin{eqnarray}\label{Psdp}
\begin{array}{rcl}
{\rm (P_{\rm SDP})}:~~ \inf &&  \left[\begin{array}{cc} A& -f\\-f^T&0 \end{array} \right] \bullet Y \cr    %
 {\rm s.t.} && \left[\begin{array}{cc}B&-g\\-g^T&-\mu \end{array} \right] \bullet Y \leq
 0,\\
 &&Y_{n+1,n+1}=1\\
 &&Y\succeq 0,
\end{array}
\end{eqnarray}
where $X\bullet  Y:={\rm trace}(X^TY)$ is the usual matrix inner
product.  Let $v(\cdot)$ denote the optimal value of problem
$(\cdot)$. By the weak duality, we have
\begin{equation}
v({\rm QP1QC})\ge v{\rm (P_{\rm SDP})} \ge v{\rm (D_{\rm
SDP})}.\label{wd}
\end{equation}

To obtain the strong duality, observe that
\begin{eqnarray*}
\begin{array}{rl}
\displaystyle{v({\rm QP1QC})=\inf_{x\in\mathbb{R}}\left\{F(x)\bigg|G(x)\leq\mu\right\}}
=\sup\left\{s\in\mathbb{R}\bigg|\left\{x\in\mathbb{R}^n|F(x)-s<0,G(x)-\mu\leq0\right\}
=\emptyset\right\}.
\end{array}
\end{eqnarray*}
Under the primal Slater condition and by the S-lemma \cite{pt}, we get 
\begin{eqnarray}\label{equation2}
\begin{array}{l}
~~~\sup\left\{s\in\mathbb{R}\bigg|\left\{x\in\mathbb{R}^n|F(x)-s<0,G(x)-\mu\leq0\right\}=\emptyset\right\}\\
=\sup\left\{s\in\mathbb{R}\bigg|\exists \sigma\ge0 {\hbox{ such that }} F(x)-s+\sigma G(x)-\sigma\mu\geq0,\ \forall x\in\mathbb{R}^n\right\}\\
=\left\{\begin{array}{rl}
\sup&s\\
\mbox{s.t.}&\left[\begin{array}{cc}
A+\sigma B&-f-\sigma g\\
-f^T-\sigma g^T&-\mu\sigma-s
\end{array}\right]
\bullet\left[\begin{array}{cc}
xx^T&x\\x^T&1
\end{array}\right]\geq0,\ \forall x\in\mathbb{R}^n\\
&\sigma\geq0
\end{array}\right.\\
=\left\{\begin{array}{rl}
\sup&s\\
\mbox{s.t.}&\left[\begin{array}{cc}
A+\sigma B&-f-\sigma g\\
-f^T-\sigma g^T&-\mu\sigma-s
\end{array}\right]
\succeq0\\
&\sigma\geq0,
\end{array}\right.
\end{array}
\end{eqnarray}
where the last equality holds because
\begin{equation*}
{\left[\begin{array}{cc}
A+\sigma B&-f-\sigma g\\
-f^T-\sigma g^T&-\mu\sigma-s
\end{array}\right]}
\bullet {\footnotesize\left[\begin{array}{cc} xx^T&x\\x^T&1
\end{array}\right]}\geq0,\ \forall x\in\mathbb{R}^n
\Leftrightarrow
{\left[\begin{array}{cc}
A+\sigma B&-f-\sigma g\\
-f^T-\sigma g^T&-\mu\sigma-s
\end{array}\right]}\succeq0.
\end{equation*}
Combining (\ref{Dsdp}), (\ref{wd}) and (\ref{equation2}) together,
we conclude that
\begin{eqnarray}\label{equation3}
v({\rm QP1QC})=v{\rm (P_{\rm SDP})}=v{\rm (D_{\rm SDP})},
\end{eqnarray}
which says that the SDP relaxation of (QP1QC) is tight. The strong
duality result also appeared, e.g., in \cite[Appendix B]{boyd}.

To obtain a rank-one optimal solution of $\rm (P_{\rm SDP})$ having
the format $X^*=\left[\left(x^*\right)^T,
1\right]^T\left[\left(x^*\right)^T, 1\right]$, Sturm and Zhang
\cite{sz} provides the following rank-one decomposition procedure.
Suppose $X^*$ is a positive semidefinite optimal matrix obtained
from solving $\rm (P_{\rm SDP})$ and the rank of $X^*$ is $r$. First
compute a rank one decomposition $X^*=\sum_{i=1}^rp_ip_i^T$. Denote
$$M(G)=\left[\begin{array}{cc}B&-g\\-g^T&-\mu \end{array} \right].$$
If $\left(p_1^TM(G)p_1\right)\left(p_i^TM(G)p_i\right)\geq0$ for all
$i=1,2,\ldots,r$, set $y=p_1$. Otherwise, set
$\displaystyle{y=\frac{p_1+\alpha p_j}{\sqrt{1+\alpha^2}}}$ where
$\left(p_1^TM(G)p_1\right)\left(p_j^TM(G)p_j\right)<0$ and $\alpha $
satisfies $\left(p_1+\alpha p_j\right)^T M(g)\left(p_1+\alpha
p_j\right)=0$. Then, $X^*=X^*-yy^T$ is an optimal solution for $\rm
(P_{\rm SDP})$ of rank $r-1$. Repeat the procedure until $X^*$ is
eventually reduced to a rank one matrix.

Clearly, by the SDP relaxation, first one lifts the (QP1QC) problem
into a linear matrix inequality in a (much) higher dimensional space
where a conic (convex) optimization problem $\rm (P_{\rm SDP})$ is
solved, then followed by a matrix rank-one decomposition procedure
to run it back to a solution of (QP1QC). The idea is neat, but it
suffers from a huge computational task should a large scale (QP1QC)
be encountered.

In contrast, the Lagrange dual approach seems to be more natural but
the analysis is often subject to serious assumptions. We can see
that the Lagrange function (\ref{Lagrange}) makes a valid lower
bound for the (QP1QC) problem only when $A+\sigma B\succeq0$.
Furthermore, the early analysis, such as those in \cite{Ben,
flsx2010, More, xfgsz, yz}, all assumed the dual Slater condition in
order to secure the strong duality result. Actually, the dual Slater
condition is a very restrictive assumption since it implies that,
see \cite{Horn}, the two matrices $A$ and $B$ can be simultaneously
diagonalizable via congruence (SDC), which itself is already very
limited:
\\[+5pt]
\textbf{Simultaneously Diagonalizable via Congruence (SDC)}: $A$ and
$B$ are simultaneously diagonalizable via congruence (SDC) if there
exists a nonsingular matrix $C$ such that both $C^TAC$ and $C^TBC$
are diagonal.

Nevertheless, solving (QP1QC) via the Lagrange dual enjoys a
computational advantage over the SDP relaxation. It has been studied
thoroughly in \cite{flsx2010} that the optimal solution $x^*$ to
(QP1QC) problems, while satisfying the dual Slater condition, is
either $\displaystyle\bar
x=\lim\limits_{\sigma\rightarrow\sigma^*}(A+\sigma B)^{-1}(f+\sigma
g)$, or obtained from $\bar x + \alpha_0 \tilde x$ where $\sigma^*$
is the dual optimal solution, $\tilde x$ is a vector in the null
space of $A+\sigma^*B$ and $\alpha_0$ is some constant from solving
a quadratic equation. For (QP1QC) problems under the SDC condition
while violating the dual Slater condition, they are either unbounded
below or can be transformed equivalently to an unconstraint
quadratic problem. Since the dual problem is a single-variable
concave programming, it is obvious that the computational cost for
$x^*$ using the dual method is much cheaper than solving a
semidefinite programming plus an additional run-down procedure.

Our paper pays a revisit to (QP1QC) trying to solve it by continuing
and extending the study beyond the dual Slater condition $A+\sigma
B\succ0$ (or SDC) into the hard case that $A+\sigma B\succeq0.$ Our
analysis was motivated by an early result by J. J. Mor$\acute{\rm
e}$ \cite[Theorem 3.4]{More} in 1993, which states:
\\[+5pt]
{\it Under the primal Slater condition and assuming that $B\neq0$, a
vector $x^*$ is a global minimizer of (QP1QC) if and only if there
exists a pair $(x^*,\sigma^*)\in R^n\times R$ such that $x^*$ is
primal feasible satisfying $x^{*T}Bx^*-2g^Tx^*\leq \mu$;
$\sigma^*\ge0$ is dual feasible; the pair satisfies the
complementarity $\sigma^*(x^{*T}Bx^*-2g^Tx^*-\mu)=0$; and a second
order condition $A+\sigma^*B \succeq 0$ holds.}
\\[+5pt]
Notice that the speciality of the statement is to obtain a superior
version of the second order necessary condition $A+\sigma^*B \succeq
0$. Normally in the context of a general nonlinear programming, one
can only conclude that $A+\sigma^*B$ is positive semidefinite on
{\it the tangent subspace} of $G(x^*)=0$ when $x^*$ lies on the
boundary. The original proof was lengthy, but it can now be
understood from the theory of semidefinite programming without much
difficulty.

To this end, we first solve (QP1QC) separately without the primal
Slater condition in Section 2. In Section 3, we establish new
results for the positive semideifinite matrix pencil
$$I_{\succeq}(A,B)=\{\sigma\in \mathbb{R}\mid ~ A+\sigma B \succeq
0\}$$ with which, in Section 4, we characterize necessary and
sufficient conditions for two unsolvable (QP1QC) cases: either
(QP1QC) is unbounded below, or bounded below but unattainable. The
conditions are not only polynomially checkable, but provide a
solution procedure to obtain an optimal solution $x^*$ of (QP1QC)
when the problem is solvable. In Section 5, we use numerical
examples to show why (QP1QC), while failing the (SDC) condition,
become the hard case. Conclusion remarks are made in section
\ref{conclusions}.

\section{(QP1QC) Violating Primal Slater Condition}\label{chapter_PSCfailed}

In this section, we show that (QP1QC) with no Slater point can be
separately solved without any condition. Notice that, if $\mu>0$,
$x=0$ satisfies the Slater condition. Consequently, if the primal
Slater condition is violated, we must have $\mu\leq0$, and
\begin{eqnarray}\label{PSCfailed}
\min_{x\in {\mathbb R}^n} G(x)=x^TBx-2g^Tx\geq\mu,\end{eqnarray}
which implies that $G(x)$ is bounded from below and thus
$$B\succeq0;~g\in \mathcal{R}(B);~G(B^+g)=-g^TB^+g\geq\mu$$
where $\mathcal{R}(B)$ is the range space of $B$, and $B^+$ is the
Moore-Penrose generalized inverse of $B$. Conversely, if
$B\succeq0$, then $G(x)$ is convex. Moreover, if
$g\in\mathcal{R}(B)$, the set of critical points
$\left\{x\big|Bx=g\right\}\neq\emptyset$, and every point in this
set assumes the global minimum $-g^TB^+g$. Suppose
$G(B^+g)=-g^TB^+g\geq\mu$, it implies that
$\left\{x\big|x^TBx-2g^Tx<\mu\right\}=\emptyset$. In other words,
$$
\mbox{(QP1QC) violates the primal Slater condition if and only if }
\left\{\begin{array}{l}
\mu\leq0;\\
B\succeq0;\\
g\in R(B);\\
-g^TB^+g\geq\mu.
\end{array}\right.
$$
In this case, the constraint set is reduced to
\begin{eqnarray*}
\left\{x\big|g(x)\leq\mu\right\}
=\left\{
\begin{array}{ll}
\emptyset,&\mbox{ if } -g^TB^+g>\mu; \\
\left\{x\big|Bx=g\right\},&\mbox{ if } -g^TB^+g=\mu.
\end{array}
\right.
\end{eqnarray*}
Namely, (QP1QC) is either infeasible (when $-g^TB^+g>\mu$) or
reduced to a quadratic programming with one linear equality
constraint. In the latter case, the following problem is to be
dealt:
\begin{eqnarray}\label{PSCfailed2}
\begin{array}{rl}
\displaystyle{\inf_{x\in\mathbb{R}^n}}&x^TAx-2f^Tx\\
\mbox{s.t.}&Bx=g.
\end{array}
\end{eqnarray}
All the feasible solutions of (\ref{PSCfailed2}) can be expressed as
$x=B^+g+Vy,\ \forall y\in\mathbb{R}^{\mbox{dim }N(B)}$ where $N(B)$
denotes the null space of $B$. The objective function becomes
\begin{eqnarray*}
\begin{array}{rl}
x^TAx-2f^Tx&=\left(B^+g+Vy\right)^TA\left(B^+g+Vy\right)
-2f^T\left(B^+g+Vy\right)\\
&=y^TV^TAVy+2\left(V^TAB^+g-V^Tf\right)^Ty+g^TB^+AB^+g-2f^TB^+g\\
&=z^T\hat{D}z+2\Lambda^Tz+\alpha,
\end{array}
\end{eqnarray*}
where $z=Q^Ty$, $Q$ is an orthonormal matrix such that
$V^TAV=Q\hat{D}Q^T$;
$\hat{D}=diag\left(\hat{d}_1,\hat{d}_2,\ldots,\hat{d}_n\right)$;
$\Lambda=Q^T\left(V^TAB^+g-V^Tf\right)$ and
$\alpha=g^TB^+AB^+g-2f^TB^+g$. Let
$$h(z)=\displaystyle{\sum_{i\in I}\Big(\hat{d}_iz_i^2
+2\Lambda_iz_i}\Big)+\alpha,\ I=\left\{1,2,\ldots,n\right\}.$$ 
Then, a feasible (QP1QC) violating the primal Slater condition can
be summarized as follows:
\begin{itemize}
\item[$\bullet$] If $\displaystyle{\min_{i\in I}\left\{\hat{d}_i\right\}<0}$ or
there is an index $i_0\in I$ such that $\hat{d}_{i_0}=0$ with
$\Lambda_{i_0}\neq0$, then (QP1QC) is unbounded from below.
\item[$\bullet$] Otherwise, $h(z)$ is convex,
$v(QP1QC)=\alpha-{\sum_{i\in J}\frac{\Lambda_i^2}{\hat{d}_i}}$
where $J=\{i\in I|\hat{d}_i>0\}$. Moreover, $x^*=B^+g+VQz$ with
\begin{eqnarray*} z_i =\left\{
\begin{array}{ll}
-\frac{\Lambda_i}{\hat{d}_i},&\mbox{ if } i\in J; \\
0,&\mbox{ if } i\in I\setminus J.
\end{array}
\right.
\end{eqnarray*}
is one of the optimal solutions to (QP1QC).

\end{itemize}

\section{New Results on Matrix Pencil and SDC}\label{matrix_pencil}
As we have mentioned before, we will use matrix pencils, that is,
one-parameter families of matrices of the form $A+\sigma B$ as a
main tool to study (QP1QC). An excellent survey for matrix pencils
can be found in \cite{Uhlig}. In this section, we first summarize
important results about matrix pencils followed by new findings of
the paper.

For any symmetric matrices $A,B$, define
\begin{eqnarray}
I_{\succ}(A,B)&:=&\{\sigma\in \mathbb{R}\mid ~ A+\sigma B \succ 0\},\\
I_{\succeq}(A,B)&:=&\{\sigma\in \mathbb{R}\mid ~ A+\sigma B \succeq 0\},\\
II_{\succ}(A,B)&:=&\{(\mu,\sigma)\in \mathbb{R}^2\mid ~ \mu A+\sigma B \succ 0\},\\
II_{\succeq}(A,B)&:=&\{(\mu,\sigma)\in \mathbb{R}^2\mid ~ \mu A+\sigma B \succeq 0\},\\
Q(A)&:=&\{v \mid~v^TAv=0\},\\
N(A)&:=&\{v \mid~ Av=0\}.
\end{eqnarray}

The first result characterizes when $I_{\succ}(A,B)\neq \emptyset$
and when $II_{\succ}(A,B) \neq \emptyset$.

\begin{thm}[\cite{Uhlig}]\label{thm:1}
If $A,B\in \mathbb{R}^{n\times n}$ are symmetric matrices, then
$I_{\succ}(A,B)\neq \emptyset$ if and only if
\begin{equation}
w^TAw>0,~\forall w\in Q(B),~ w\neq 0. \label{cond:1}
 \end{equation}
Furthermore, $II_{\succ}(A,B) \neq \emptyset$ implies that
\begin{equation}
 Q(A)\cap  Q(B)=\{0\}, \label{cond:2}
\end{equation}
which is also sufficient if $n\geq 3$.
\end{thm}

The second result below extends the discussion in Theorem
\ref{thm:1} to the positive semidefinite case $I_{\succeq}(A,B)\neq
\emptyset$ assuming that $B$ is indefinite. This result was also
used by Mor$\acute{\rm{e}}$ to solve (QP1QC) in \cite{More}. An
example shown in \cite{More} indicates that the assumption about the
indefiniteness of $B$ is critical.

\begin{thm}[\cite{More}]
If $A,B\in R^{n\times n}$ are symmetric matrices and $B$ is indefinite,
then $I_{\succeq}(A,B)\neq \emptyset$
if and only if
\begin{equation}
  w^TAw\geq 0,~\forall w\in Q(B). \label{cond:3}
  \end{equation}
\end{thm}

Other known miscellaneous results related to our discussion for
(QP1QC) are also mentioned below. In particular, the proof of Lemma
\ref{thm:lem:2} is constructive.

\begin{lem}\cite{More} \label{thm:lem:1}
Both $I_{\succ}(A,B)$ and $I_{\succeq}(A,B)$ are intervals.
\end{lem}

\begin{lem}\cite{Horn,JB-Urruty} \label{thm:lem:2}
$II_{\succ}(A,B)=\{(\lambda,\nu)\in R^2\mid  \lambda A+\nu B\succ
0\}\neq \emptyset$ implies that $A$ and $B$ are simultaneously
diagonalizable via congruence(SDC). When $n\ge3$,
$II_{\succ}(A,B)\not=\emptyset$ is also necessary for $A$ and $B$ to
be SDC.
\end{lem}

In the following we discuss necessary conditions for
$I_{\succeq}(A,B)\neq \emptyset$ (Theorem \ref{thm:3}) and
sufficient conditions for $II_{\succeq}(A,B)\neq \emptyset$ (Theorem
\ref{thm:3-sufficient}) without the indefiniteness assumption of
$B$.

\begin{thm}\label{thm:3}
Let $A,B\in R^{n\times n}$ be symmetric matrices and suppose
$I_{\succeq}(A,B)\neq \emptyset$. If $I_{\succeq}(A,B)$ is a
single-point set, then $B$ is indefinite and
$I_{\succ}(A,B)=\emptyset$. Otherwise, $I_{\succeq}(A,B)$ is an
interval and
\begin{itemize}
\item (a) $  Q(A)\cap  Q(B)=N(A)\cap N(B)$.
\item (b) $N(A+\sigma B)=N(A)\cap N(B)$, for any $\sigma$ in the interior of $I_{\succeq}(A,B)$.
\item (c) $A, B$ are simultaneously diagonalizable via congruence.
\end{itemize}
\end{thm}

\begin{proof}
Suppose $I_{\succeq}(A,B)=\{\sigma\}$ is a single-point set. Then
for any $\sigma_1<\sigma<\sigma_2$, neither $A+\sigma_1 B$ nor
$A+\sigma_2 B$ is positive semidefinite. There are vectors $v_1$ and
$v_2$ such that
\begin{eqnarray*}
&&0>v_1^T(A+\sigma_1 B)v_1 = v_1^T(A+\sigma B)v_1+(\sigma_1-\sigma)v_1^TBv_1\geq (\sigma_1-\sigma)v_1^TBv_1; \\
&&0>v_2^T(A+\sigma_2 B)v_2 = v_2^T(A+\sigma B)v_2+(\sigma_2-\sigma)v_2^TBv_2\geq (\sigma_2-\sigma)v_2^TBv_2.
\end{eqnarray*}
It follows that
\[
v_1^TBv_1>0 ~\hbox{\rm and~}
 v_2^TBv_2<0,
\]
so $B$ is indefinite. Furthermore, suppose there is a $\hat\sigma\in
I_{\succ}(A,B)$ such that $A+\hat\sigma B\succ 0$. Then, $A+\sigma
B\succ0$ for $\sigma\in(\hat\sigma-\delta,\hat\sigma+\delta)$ and
some $\delta>0$ sufficiently small, which contradicts to the fact
that $I_{\succeq}(A,B)$ is a singleton.

Now assume that $I_{\succeq}(A,B)$ is neither empty nor a
single-point set. According to Lemma \ref{thm:lem:1},
$I_{\succeq}(A,B)$ is an interval. 
To prove (a), choose
$\sigma_1,\sigma_2\in I_{\succeq}(A,B)$ with $\sigma_1\neq\sigma_2$.
If $x\in Q(A)\cap Q(B)$, we have $x^T\left(A+\sigma_1B\right)x=0$.
Since $A+\sigma_1B\succeq0$, $x^T\left(A+\sigma_1B\right)x=0$
implies that $\left(A+\sigma_1B\right)x=0$.
%
%
Similarly, we have $\left(A+\sigma_2B\right)x=0$. Therefore,
$$0=\left(A+\sigma_1B\right)x-\left(A+\sigma_2B\right)x
=\left(\sigma_1-\sigma_2\right)Bx,$$ indicating that $Bx=0$. It follows
also that $Ax=0$. In other words, $Q(A)\cap Q(B)\subseteq N(A)\cap
N(B)$. Since it is trivial to see that $N(A)\cap N(B)\subseteq
Q(A)\cap Q(B)$, we have the statement (a).

To prove (b), choose $\sigma_1, \sigma, \sigma_2\in
int\left(I_{\succeq}\left(A,B\right)\right)$ such that
$\sigma_1<\sigma<\sigma_2.$ Then, for all $x\in R^n$,
$$\begin{array}{c}
0\leq x^T\left(A+\sigma_1B\right)x
=x^T\left(A+\sigma B\right)x+\left(\sigma_1-\sigma\right)x^TBx,\\
0\leq x^T\left(A+\sigma_2B\right)x
=x^T\left(A+\sigma B\right)x+\left(\sigma_2-\sigma\right)x^TBx,
\end{array}$$
which implies that
$$0\leq x^TBx\leq0,$$
and hence $x^TBx=0.$ Moreover, if $x\in N\left(A+\sigma B\right)$,
we have not only $\left(A+\sigma B\right)x=0,$ but
$x^T\left(A+\sigma B\right)x=0$. Since $x^TBx=0$, there also is
$x^TAx=0$. By the statement (a), we have proved that
$$N\left(A+\sigma B\right)\subseteq Q(A)\cap Q(B)=N(A)\cap N(B).$$
With the fact that $N(A)\cap N(B)\subseteq N\left(A+\sigma B\right)$
for $\sigma\in int\left(I_{\succeq}\left(A,B\right)\right)$, we find
that $N\left(A+\sigma B\right)$ is invariant for any $\sigma\in
int\left(I_{\succeq}\left(A,B\right)\right)$, and the statement (b)
follows.

For the statement (c), let $V\in R^{n\times r}$ be the basis matrix
of $N(A+\sigma B)$ for some $\sigma\in
int\left(I_{\succeq}\left(A,B\right)\right)$, where $r={\rm
dim}(N(A+\sigma B))$. 
Extend $V$ to a nonsingular matrix $[U ~ V]\in R^{n\times n}$ such
that $U^TU$ is nonsingular and $U^TV=0$. Then,
\[
\left[\begin{array}{c}U^T\\V^T\end{array}\right](A+\sigma B)\left[U~V\right]
=\left[\begin{array}{cc}U^T(A+\sigma B)U& U^T(A+\sigma B)V\\V^T(A+\sigma B)U&V^T(A+\sigma B)V\end{array}\right]
=\left[\begin{array}{cc}U^TAU+\sigma U^T BU& 0\\0&0\end{array}\right].
\]
Consequently,
\begin{equation}\label{UTAU}
U^TAU+\sigma U^T BU\succ 0,
\end{equation}
which can be verified by applying $u^T$ and $u$, $u\in R^{n-r}$, to
$U^TAU+\sigma U^T BU$. Then, $(u^TU^T)(A+\sigma B)(Uu)=0$ if and
only if $(A+\sigma B)(Uu)=0$. Since $U$ is the orthogonal complement
of $V$, it must be $u=0$.

By Lemma \ref{thm:lem:2}, $U^TAU$ and $U^TBU$ are simultaneously
diagonalizable via congruence, i.e., there exists nonsingular square
matrix $U_1$ such that both $U_1^TU^TAUU_1$ and $U_1^TU^TBUU_1$ are
diagonal. Let $W$ be such that
\[\left[\begin{array}{cc}U_1&0\\0&W\end{array}\right]
\]
is nonsingular. By the statement (b), $V$ also spans $N(A)\cap N(B)$
and thus $AV=BV=0$. It indicates that
\begin{eqnarray*}
\left[\begin{array}{cc}U_1^T&0\\0&W^T\end{array}\right]\left[\begin{array}{c}U^T\\V^T\end{array}\right]A\left[\begin{array}{cc}U&V\end{array}\right]
\left[\begin{array}{cc}U_1&0\\0&W\end{array}\right]&=&\left[\begin{array}{cc}U_1^TU^TAUU_1&0\\0&0\end{array}\right],\\
\left[\begin{array}{cc}U_1^T&0\\0&W^T\end{array}\right]\left[\begin{array}{c}U^T\\V^T\end{array}\right]B\left[\begin{array}{cc}U&V\end{array}\right]
\left[\begin{array}{cc}U_1&0\\0&W\end{array}\right]&=&\left[\begin{array}{cc}U_1^TU^TBUU_1&0\\0&0\end{array}\right].
\end{eqnarray*}
That is, $A$ and $B$ are simultaneously diagonalizable via
congruence. The statement (c) is proved.
\end{proof}

\begin{thm}\label{thm:3-sufficient}
Let $A,B\in R^{n\times n}$ be symmetric matrices. Suppose $ Q(A)\cap
Q(B)=N(A)\cap N(B)$ and
\begin{equation}
{\rm dim}\left\{N(A)\cap N(B)\right\}\leq n-3.\label{thm:eq:0}
\end{equation}
Then,
\begin{itemize}
\item (a') $II_{\succ}(U^TAU,U^TBU)\neq \emptyset$ where $U$ is the basis
matrix spans the orthogonal complement of $N(A)\cap N(B)$;
\item (b') $II_{\succeq}(A,B)\neq \emptyset$; and
\item (c') $A, B$ are simultaneously diagonalizable via congruence.
\end{itemize}
\end{thm}

\begin{proof}
Following the same notation in the proof of Theorem \ref{thm:3}, we
let $V\in R^{n\times r}$ be the basis matrix of $N(A)\cap N(B)$,
$r={\rm dim}(N(A)\cap N(B))$, and then extend $V$ to a nonsingular
matrix $[U ~ V]\in R^{n\times n}$ such that $U^TU$ is nonsingular
and $U^TV=0$. For any $x\in Q(U^TAU)\cap Q(U^TBU)$, we have $
x^TU^TAUx=x^TU^TBUx=0$ and thus $Ux\in Q(A)\cap Q(B).$ By
assumption, $ Q(A)\cap Q(B)=N(A)\cap N(B)$, there must exist $z\in
R^r$ such that
\[
Ux=Vz.
\]
Therefore,
\[
x=(U^TU)^{-1}U^TVz=0,
\]
which shows that
\begin{equation}\label{thm:eq:5}
 Q(U^TAU)\cap Q(U^TBU)=\{0\}.
\end{equation}
It follows from Assumption (\ref{thm:eq:0}) that
$${\rm dim}(U^TAU)={\rm dim}(U^TBU)=n-r\ge3.$$
According to Theorem \ref{thm:1}, (\ref{thm:eq:5}) implies that
\begin{equation}
II_{\succ}(U^TAU,U^TBU)\neq \emptyset. \label{thm:eq:6}
\end{equation}
Namely, there exists $(\mu,\sigma)\in \mathbb{R}^2$ such that
$\mu(U^TAU)+\sigma(U^TBU)\succ0$, which proved the statement (a').
Now, for $y\in R^n$, we can write $y=Uu+Vv$, where $u\in R^{n-r},
v\in R^{r}$. Then,
\begin{eqnarray*}
&&y^T(\mu A +\sigma B)y\\
&=&(u^TU^T+v^TV^T)(\mu A +\sigma B)(Uu+Vv)\\
&=&u^T(\mu U^TAU+\sigma U^TBU)u\ge0,
\end{eqnarray*}
which proves $II_{\succeq}(A,B)\neq \emptyset.$ From
(\ref{thm:eq:6}) and Lemma \ref{thm:lem:2}, we again conclude that
$U^TAU$ and $U^TBU$ are simultaneously diagonalizable via
congruence. It follows from the same proof for the statement (c) in
Theorem \ref{thm:3} that $A, B$ are indeed simultaneously
diagonalizable via congruence. The proof is complete.
\end{proof}

\begin{rem}
When $Q(A)\cap Q(B)=\{0\}$, the assumption of Theorem
\ref{thm:3-sufficient} becomes $N(A)\cap N(B)=\{0\}$. In this case,
the dimension condition (\ref{thm:eq:0}) is equivalent to $n\geq 3$
and the basis matrix $U$ is indeed $I_n$, the identical matrix of
dimension $n$. From the statement (a') in Theorem
\ref{thm:3-sufficient}, we conclude that $II_{\succ}(A,B)\neq
\emptyset$. In other words, we have generalized the part of
conclusion in Theorem \ref{thm:1} that ``$Q(A)\cap Q(B)=\{0\}$ and
$n\ge3$'' implies ``$II_{\succ}(A,B)\neq \emptyset$''.
\end{rem}

As an immediate corollary of Theorem \ref{thm:3-sufficient}, we
obtain sufficient conditions to make $A$ and $B$ simultaneously
diagonalizable via congruence.
\begin{cor}\label{cor}
Let $A,B\in R^{n\times n}$ be symmetric matrices. Suppose one of the
following conditions is satisfied,
\begin{itemize}
\item {{\rm (S1)}} $I_{\succeq}(A,B)$ contains more than one point;
\item {{\rm (S2)}} $I_{\succeq}(B,A)$ contains more than one point;
\item {{\rm (S3)}} $ Q(A)\cap  Q(B)=N(A)\cap N(B)$, with ${\rm dim}\left\{N(A)\cap N(B)\right\}\neq n-2$,
\end{itemize}
then $A, B$ are simultaneously diagonalizable via congruence.
\end{cor}
\begin{proof} Conditions (S1) and (S2) directly implies that $I_{\succeq}(A,B)$
(and $I_{\succeq}(B,A)$ respectively) is non-empty and is not a
single point set. The conclusion follows from the statement (c) of
Theorem \ref{thm:3}. Condition (S3) is a slight generalization of
Theorem \ref{thm:3-sufficient} where the situation for ${\rm
dim}\left\{N(A)\cap N(B)\right\}\leq n-3$ has already been proved.
It remains to show that $A, B$ are simultaneously diagonalizable via
congruence when ${\rm dim}\left\{N(A)\cap N(B)\right\}=n\vee (n-1).$
Notice that ${\rm dim}\left\{N(A)\cap N(B)\right\}=n$ implies
$A=B=0$ (in which case the conclusion follows immediately), whereas
${\rm dim}\left\{N(A)\cap N(B)\right\}=n-1$ implies that matrices
$A$ and $B$ must be expressed as $A=\alpha uu^T,~B=\beta uu^T$ for
some vector $u\in R^n$ and $\alpha^2+\beta^2\not=0$. Suppose
$\beta\not=0$, then $A=\frac{\alpha}{\beta}B$ so the two matrices
are certainly simultaneously diagonalizable via congruence.
\end{proof}

Example \ref{exam:1} below shows that the non-single-point
assumption of $I_{\succeq}(A,B)$ is necessary to assure statement
(a) and (c); and the dimensionality assumption in (S3) is also
necessary.

\begin{exam}\label{exam:1}
Consider
\begin{equation}
A=\left[\begin{array}{cc}1&0\\0&0\end{array}\right],~B=\left[\begin{array}{cc}0&1\\1&0\end{array}\right].
\end{equation}
Then there is only one $\sigma=0$ such that $A+\sigma B\succeq 0$, but
\begin{eqnarray*}
 Q(A)\cap  Q(B)&=&\left\{\left[\begin{array}{c}0\\  \ast \end{array}\right]\right\};\\
N(A)\cap N(B)&=&\left\{0\right\};\\
{\rm dim}\left\{N(A)\cap N(B)\right\}&=&n-2.
\end{eqnarray*}
It is not difficult to verify that here $A$ and $B$ cannot be
simultaneously diagonalizable via congruence.
\end{exam}

\begin{rem}
Corollary \ref{cor} extends Lemma \ref{thm:lem:2} since $
\{(\lambda,\nu)\in R^2\mid  \lambda A+\nu B\succ 0\}\neq \emptyset $
implies that either $I_{\succ}(A,B)\neq \emptyset$ or
$I_{\succ}(B,A)\neq \emptyset$, belonging to (S1) and (S2),
respectively.
\end{rem}
\begin{rem}
Both (S1) and (S2) are polynomially checkable. In literature, there
are necessary and sufficient conditions for $A$ and $B$ to be
simultaneously diagonalizable via congruence, see \cite{Bec} and
references therein. But none is polynomially checkable.
\end{rem}

%

\section{Solving (QP1QC) via Matrix Pencils}\label{chapter_PSC}

%

In this section, we use matrix pencils developed in the previous
section to establish necessary and sufficient conditions for the
``undesired'' (QP1QC) which are unbounded below, and which are
bounded below but unattainable. All the discussions here are subject
to the assumption of the primal Slater condition since the other
type of (QP1QC) have been otherwise solved in Section 2 above.
Interestingly, probing into the two undesired cases eventually lead
to a complete characterization for the solvability of (QP1QC).

\begin{thm} \label{thm:alg:3} Under the primal Slater condition,
(QP1QC) is unbounded below if and only
if the system of $\sigma$:
\begin{equation}\label{sdp:4}
\left\{
\begin{array}{l}
A+\sigma B \succeq 0,\\
\sigma\geq 0,\\
f+\sigma  g\in \mathcal{R}(A+\sigma B),
\end{array}\right.
\end{equation}
has no solution.
\end{thm}
\begin{proof}
Since the primal Slater condition is satisfied, if the infimum of
(QP1QC) is finite, by the strong duality \cite{Helmberg02}, $v({\rm
P_{SDP}})=v({\rm D_{SDP}})$ and the optimal value is attained for
(${\rm D_{SDP}}$). Then, the system (\ref{sdp:4}) has a solution.
Conversely, suppose the system (\ref{sdp:4}) has a solution. Then
(${\rm D_{SDP}}$) is feasible which provides a finite lower bound
for (${\rm P_{SDP}}$). As the primal Slater condition is satisfied,
(QP1QC) is feasible and must be bounded from below.
\end{proof}


\begin{rem} Checking whether the system (\ref{sdp:4}) has a solution can be done
in polynomial time. To this end, denote
\begin{eqnarray}
\sigma_l^*:=\min_{A+\sigma B\succeq 0} \sigma, \label{sdp:p1}\\
 \sigma_u^*:=\max_{A+\sigma B\succeq 0} \sigma, \label{sdp:p2}
\end{eqnarray}
and let $V\in R^{n\times r}$ be the basis matrix of $N(A)\cap N(B)$
and $U\in R^{n\times (n-r)}$ satisfy $[U~V]$ is nonsingular and
$U^TV=0$. From
\[
\left[\begin{array}{c}U^T\\V^T\end{array}\right](A+\sigma B)\left[U~V\right]
=\left[\begin{array}{cc}U^T(A+\sigma B)U& U^T(A+\sigma B)V\\V^T(A+\sigma B)
U&V^T(A+\sigma B)V\end{array}\right]
=\left[\begin{array}{cc}U^TAU+\sigma U^T BU& 0\\0&0\end{array}\right],
\]
it can be seen that, if $U^T BU=0$ and $U^TAU\not\succeq 0$, then
$I_{\succeq}(A,B)=\emptyset$ and the system (\ref{sdp:4}) has no
solution. It follows that (QP1QC) must be unbounded below in this
case.

Assume that $I_{\succeq}(A,B)\not=\emptyset$ and $\bar\sigma\in
I_{\succeq}(A,B)$. Then, we observe from
$$A+\sigma B=(A+\bar\sigma B)+(\sigma - \bar\sigma)B$$
that $\sigma_l^*=-\infty$ if and only if $B\preceq 0$. Similarly,
$\sigma_u^*=\infty$ if and only if $B\succeq0$.

Now suppose $-\infty<\sigma_l^*\le\sigma_u^*<\infty$. From the proof
for the statement (c) in Theorem \ref{thm:3} and (\ref{UTAU}), if
$\bar\sigma\in int\left(I_{\succeq}\left(A,B\right)\right)$, then
$U^TAU+\bar\sigma U^T BU\succ 0$. Conversely, if $U^TAU+\bar\sigma
U^T BU\succ 0$, there must be some $\epsilon>0$ such that
$U^TAU+(\bar\sigma\pm\epsilon) U^T BU\succ 0$ and thus
$A+(\bar\sigma\pm\epsilon)B\succeq 0$. It holds that $\bar\sigma\in
int\left(I_{\succeq}\left(A,B\right)\right)$. Consequently,
$\sigma_l^*$ and $\sigma_u^*$, as end points of
$I_{\succeq}\left(A,B\right)$, must be roots of the polynomial
equation
\begin{equation}
{\rm det}(U^TAU+\sigma U^T BU)=0.\label{poly}
\end{equation}
We first find all the roots of the polynomial equation (\ref{poly}),
denoted by $r_1,\ldots,r_{n-r}$, which can be done in polynomial
time. Then,
\begin{eqnarray*}
&&\sigma_l^*=\min\{r_i:~U^TAU+r_i U^T BU\succeq 0, ~i=1,\ldots,n-r\},\\
&&\sigma_u^*=\max\{r_i:~U^TAU+r_i U^T BU\succeq 0,~ i=1,\ldots,n-r\}.
\end{eqnarray*}

\begin{itemize}
\item {Case (a)} ~ If $\sigma_u^*<0$, (\ref{sdp:4}) has no solution.
\item {Case (b)} ~ $\sigma_u^*=0$ or $\sigma_l^*=\sigma_u^*>0$. Checking the feasibility of
(\ref{sdp:4}) is equivalent to verify whether the linear equation
\[
(A+\sigma B)x=f+\sigma g
\]
has a solution for $\sigma=0$ or $\sigma=\sigma_l^*=\sigma_u^*$.
\item {Case (c)} ~ $\sigma_u^*>0$ and $\sigma_l^*<\sigma_u^*$. According to
Theorem \ref{thm:3}, the basis matrix $V$ spans $N(A+\sigma'
B)=N(A)\cap N(B)$ for any $\sigma'\in (\sigma_l^*,\sigma_u^*).$
\begin{itemize}
\item {Subcase (c1)} ~$\sigma_l^*< 0$.  (\ref{sdp:4}) has a solution if and only if either
there is a $\sigma\in [0,\sigma_u^*)$ such that
\begin{equation}\label{sigma:eq1}
V^Tf+\sigma V^Tg=0,
\end{equation}
or the linear equation
\[
(A+\sigma_u^* B)x=f+\sigma_u^* g
\]
has a solution.
\item {Subcase (c2)} ~$\sigma_l^*\geq  0$.  (\ref{sdp:4}) has a solution if and only if
one of the following three conditions holds:
\begin{eqnarray}
&&V^Tf+\sigma V^Tg=0,~{\rm for ~some } ~\sigma\in (\sigma_l^*,\sigma_u^*);\label{sigma:eq2}\\
&&(A+\sigma_l^* B)x=f+\sigma_l^* g ~{\rm has ~a ~solution};\nonumber\\
&&(A+\sigma_u^* B)x=f+\sigma_u^* g ~{\rm has ~a ~solution}.\nonumber
\end{eqnarray}
\end{itemize}
\end{itemize}
We show either (\ref{sigma:eq1}) or (\ref{sigma:eq2}) is
polynomially solvable. Since $V,f,g$ are fixed, if $V^Tf=V^Tg=0$,
$V^Tf+\sigma V^Tg=0$ holds for any $\sigma$. If $V^Tf$ and $V^Tg\neq
0$ are linear dependent, there is a unique $\sigma$ such that
$V^Tf+\sigma V^Tg=0$. Otherwise, $V^Tf+\sigma V^Tg=0$ has no
solution.
\end{rem}

From Theorem \ref{thm:alg:3}, it is clear that, if (QP1QC) is
bounded from below, the positive semidefinite matrix pencil
$I_{\succeq}(A,B)$ must be non-empty. In solving (QP1QC), we
therefore divide into two cases: $I_{\succeq}(A,B)$ is an interval
or $I_{\succeq}(A,B)$ is a singleton.

\begin{thm} \label{thm:alg:4-interval}
Under the primal Slater condition, if the optimal value of (QP1QC)
is finite and $I_{\succeq}(A,B)$ is an interval, the infimum of
(QP1QC) is always attainable.
\end{thm}
\begin{proof} Let $[U~ V]\in R^{n\times n}$ be a nonsingular matrix such that $V$
spans $N(A)\cap N(B)$. By Theorem \ref{thm:3}, for any $\sigma'\in
int\left(I_{\succeq}\left(A,B\right)\right)$, $V$ also spans
$N(A+\sigma' B)$. It follows from (\ref{UTAU}) that
\begin{equation}\label{case1:1}
U^TAU+\sigma' U^TBU \succ 0,~\forall\sigma'\in int\left(I_{\succeq}\left(A,B\right)\right).
\end{equation}

Since $v(QP1QC)>-\infty$, there is a solution $\sigma$ to the system
(\ref{sdp:4}) such that $\sigma\in I_{\succeq}(A,B)\cap [0,+\infty)$
and $f+\sigma g\in R(A+\sigma B)$. Since $V$ spans $N(A)\cap N(B)$,
$V$ must span a subspace contained in $N(A+\sigma B)$. It follows
from $f+\sigma g\in R(A+\sigma B)$ that
\begin{equation}\label{case1:2}
V^Tf+\sigma V^Tg=0.
\end{equation}

Consider the nonsingular transformation $x=Uu+Vv$ which splits
(QP1QC) into $u$-part and $v$-part:
\begin{eqnarray}
& \min &f(u,v)= u^TU^TAUu-2f^TUu-2f^TVv \label{case1:3}\\
&{\rm s.t.} & g(u,v)=u^TU^TBUu-2g^TUu-2g^TVv\le \mu.\label{case1:4}
\end{eqnarray}
\begin{itemize}
\item{Case (d)}  $V^Tf=V^Tg=0$. Then, (QP1QC) by (\ref{case1:3})-(\ref{case1:4})
is reduced to a smaller (QP1QC) having only the variable $u$.
Moreover, the smaller (QP1QC) satisfies the dual Slater
condition due to (\ref{case1:1}). It is hence solvable. See, for
example, \cite{More,yz,flsx2010}.
%
%
%
\item{Case (e)}  $V^Tf=0,~V^Tg\neq 0$. By (\ref{case1:2}), we have $\sigma=0$
which implies that $A\succeq 0$ and $f\in R(A)$. In this case,
since (\ref{case1:4}) is always feasible for any $u$,
(\ref{case1:3})-(\ref{case1:4}) becomes an unconstrained convex
quadratic program (\ref{case1:3}), which is always solvable due
to $f\in R(A)$.
\item{Case (f)} $V^Tf \neq 0$. By (\ref{case1:2}), $V^Tg \neq 0$, $\sigma$
is unique and nonzero. Thus, $\sigma>0$. Then, the constraint
(\ref{case1:4}) is equivalent to
    \[
    \sigma u^TU^TBUu-2\sigma g^TUu-2\sigma g^TVv\le\sigma\mu,
    \]
    or equivalently, by (\ref{case1:2}),
    \begin{equation}\label{case1:5}
    -2f^TVv=2\sigma g^TVv\ge \sigma u^TU^TBUu-2\sigma g^TUu-\sigma\mu.
    \end{equation}
    Therefore,  (\ref{case1:3})-(\ref{case1:4}) has a lower
    bounding quadratic problem with only the variable $u$ such
    that
    \begin{equation}\label{case1:6}
    f(u,v)\ge u^T(U^TAU+\sigma U^TBU)u-2(f^TU+\sigma g^TU)u-\sigma\mu.
    \end{equation}
    Since $A+\sigma B\succeq 0$, the right hand side of
    (\ref{case1:6}) is a convex unconstrained problem. With
    $f+\sigma g\in R(A+\sigma B)$, if $f+\sigma g=(A+\sigma
    B)w$, we can write $w=Uy^*+Vz^*$ to make $f+\sigma
    g=(A+\sigma B)Uy^*$ and thus obtain $U^T(f+\sigma g)\in
    R(U^TAU+\sigma U^TBU)$ with $y^*$ being optimal to
    (\ref{case1:6}). Substituting $u=y^*$ into (\ref{case1:5}),
    due to $V^Tf \neq 0$, there is always some $v=z^*$ such that
    $(u,v)=(y^*,z^*)$ makes (\ref{case1:5}) hold as an equality.
    Then, $g(y^*,z^*)\le\mu$ and $f(y^*,z^*)$ attains its lower
    bound in (\ref{case1:6}). Therefore, $(u,v)=(y^*,z^*)$
    solves (\ref{case1:3})-(\ref{case1:4}).
\end{itemize}

  Summarizing cases (d), (e) and (f), we have thus completed the proof of this theorem.
\end{proof}

What remains to discuss is when the positive semidefinite matrix
pencil $I_{\succeq}(A,B)$ is a singleton $\{\sigma^*\}$. By Theorem
\ref{thm:alg:3}, (QP1QC) is bounded from below if and only if we
must have $\sigma^*\ge0$ such that, with $V$ being the basis matrix
of $N(A+\sigma^* B)$ and $r={\rm dim}N(A+\sigma^* B)$, the set
\begin{equation}\label{S}
S=\{(A+\sigma^* B)^+(f+\sigma^* g) + Vy\mid~ y\in R^r\}
\end{equation} is non-empty. According to the result (quoted in Introduction)
by Mor$\acute{\rm{e}}$ in \cite{More}, $x^*$ solves (QP1QC) if and
only if $x^*\in S$ is feasible and $(x^*,\sigma^*)$ satisfy the
complementarity. We therefore have the following characterization
for a bounded below but unattainable (QP1QC).

\begin{thm} \label{thm:alg:4}
Suppose the primal Slater condition holds and the optimal value of
(QP1QC) is finite. The infimum of (QP1QC) is unattainable if and
only if $I_{\succeq}(A,B)$ is a single-point set $\{\sigma^*\}$ with
$\sigma^*\ge0$, and the quadratic (in)equality
\begin{equation}\label{soluset}
\left\{\begin{array}{cl}
g\left((A+\sigma^* B)^+(f+\sigma^* g)+Vy\right)= \mu,& {\rm if}  ~\sigma^*>0, \\
g\left( A^+ f+Vy \right)\leq \mu,& {\rm if} ~\sigma^*=0,
\end{array}\right.
\end{equation}
has no solution in $y$, where $V$ is the basis matrix of
$N(A+\sigma^* B)$.
\end{thm}
\begin{proof}
Since $v(QP1QC)$ is assumed to be finite, by Theorem
\ref{thm:alg:3}, the system (\ref{sdp:4}) has at least one solution.
By Theorem \ref{thm:alg:4-interval}, the problem can be unattainable
only when $I_{\succeq}(A,B)$ is a single-point set $\{\sigma^*\}$
with $\sigma^*\ge0$. In addition, the set $S$ in (\ref{S}) can not
have a feasible solution which, together with $\sigma^*$, satisfies
the complementarity. In other words, the system (\ref{soluset}) can
not have any solution in $y$, which proves the necessity of the
theorem. The sufficient part of the theorem which guarantees a
bounded below but unattainable (QP1QC) is almost trivial.
\end{proof}

\begin{rem} Under the condition that $I_{\succeq}(A,B)$ is a single-point set, that is,
$\sigma_l^*=\sigma_u^*$ in (\ref{sdp:p1})-(\ref{sdp:p2}), we show
how (\ref{soluset}) can be checked and how (QP1QC) can be solved in
polynomial time.

\begin{itemize}
\item {Case(g)} ~$\sigma_l^*=\sigma_u^*=0$. The set $S$ in
(\ref{S}) can be written as:
\begin{equation}\label{sol}
x^*=A^+f +Vy.
\end{equation}
The constraint $x^{*T}Bx^*-2g^Tx^*\leq \mu$ on the set $S$ is
expressed as
\begin{equation}\label{sol:1}
y^TV^TBVy+2[f^TA^+BV-g^TV]y\leq \widetilde{\mu},
\end{equation}
where \[ \widetilde{\mu}=\mu-f^TA^+BA^+f+2g^TA^+f.
\]
Then, (\ref{soluset}) has a solution if and only if
\begin{equation}\label{sol:2}
\widetilde{\mu}\geq L^*=\min_y ~y^TV^TBVy+2[f^TA^+BV-g^TV]y.
\end{equation}
Namely, if $\widetilde{\mu}< L^*$, the (QP1QC) is unattainable.
Otherwise, we show how to get an optimal solution of (QP1QC).
\begin{itemize}
\item {Subcase (g1)}~
$L^*>-\infty$. It happens only when $V^TBV\succeq 0$ and
$V^TBA^+f-V^Tg\in R(V^TBV)$. Therefore,
\[y^*:=-(V^TBV)^+(V^TBA^+f-V^Tg)
\]
is a solution to both (\ref{sol:2}) and (QP1QC).
\item {Subcase (g2)} ~$L^*=-\infty$. It happens when either
$V^TBV$ has a negative eigenvalue with a corresponding
eigenvector $\widetilde{y}$ or there is a vector
$\widetilde{y}\in N(V^TBV)$ such that
\[{\widetilde{y}}^{T}(V^TBA^+f-V^Tg)<0.
\]
Then, for any sufficient large number $k$,
\[
y^*=k\widetilde{y}
\] is feasible to (\ref{sol:1}) and thus solves the (QP1QC)
\end{itemize}
\item Case(h) ~$\sigma_l^* = \sigma_u^*> 0.$ (QP1QC) is attainable if
the constraint is satisfied as an equality by some point in $S$.
Restrict the equality constraint $G(x)=\mu$ to the set $S$
yields
%
\begin{equation}\label{sol:3}
y^TV^TBVy+2[(f+\sigma_l^* g)^T(A+\sigma_l^*B)^+BV-g^TV]y=
\widetilde{\mu},
\end{equation}
where \[ \widetilde{\mu}=\mu-(f+\sigma_l^*
g)^T(A+\sigma_l^*B)^+B(A+\sigma_l^*B)^+(f+\sigma_l^*
g)+2g^T(A+\sigma_l^*B)^+(f+\sigma_l^* g).
\]
Consider the unconstrained quadratic programming problems:
\begin{equation}\label{sol:4}
L^*/U^*=\inf/\sup ~y^TV^TBVy+2[(f+\sigma_l^*
g)^T(A+\sigma_l^*B)^+BV-g^TV]y.
\end{equation}
We claim that (\ref{soluset}) (as well as (QP1QC)) has a
solution if and only if $\widetilde{\mu}\in[L^*,U^*]$. We only
have to prove the ``if'' part, so $\widetilde{\mu}\in[L^*,U^*]$
is now assumed. Notice that $L^*>-\infty$ ($U^*<+\infty$) if and
only if $V^TBV\succeq (\preceq) 0$ and $
V^TB(A+\sigma_l^*B)^+(f+\sigma_l^*g)-V^Tg\in R(V^TBV)$.
\begin{itemize}
    \item {Subcase (h1)} ~ $L^*>-\infty$, $U^*<+\infty$. This is a trivial case
    as it happens  if and only if $V^TBV=0$,
    $V^TB(A+\sigma_l^*B)^+(f+\sigma_l^*g)-V^Tg=0$.
    Consequently, $L^*=U^*=0$ and any vector $y$ is a
    solution to (\ref{soluset}) and solves (QP1QC).
    \item {Subcase (h2)} ~ $L^*>-\infty$, $U^*=+\infty$.
    The infimum of (\ref{sol:4}) is attained at
    \[\widehat{y}=-(V^TBV)^+[V^TB(A+\sigma_l^*B)^+(f+\sigma_l^*g)-V^Tg].\]
    Furthermore, either $V^TBV$ has a positive eigenvalue
    with a corresponding eigenvector $\widetilde{y}$; or
    there is a vector $\widetilde{y}\in N(V^TBV)$ such that
    \[{\widetilde{y}}^{T}[V^TB(A+\sigma_l^*B)^+(f+\sigma_l^*g)-V^Tg]>0.\]
    Starting from $\widehat{y}$ and moving along the
    direction $\widetilde{y}$, we find that
        \[
        h(\alpha):= (\widehat{y}+\alpha \widetilde{y})^TV^TBV(\widehat{y}+\alpha \widetilde{y})+2[(f+\sigma_l^*
g)^T(A+\sigma_l^*B)^+BV-g^TV](\widehat{y}+\alpha \widetilde{y})
        \]
        is a convex quadratic function in the parameter
        $\alpha\in R$ with $h(0)=L^*$. The range of
        $h(\alpha)$ must cover all the values above $L^*$,
        particularly the value
        $\widetilde{\mu}\in[L^*,U^*]$. Then, the quadratic
        equation
        \[h(\alpha)=\widetilde{\mu}\]
        has a root at $\alpha^*\in(0,+\infty)$ which
        generates a solution $y^*=\widehat{y}+\alpha^*
        \widetilde{y}$ to (\ref{sol:3}).
 \item {Subcase (h3)} ~ $L^*=-\infty$, $U^*<+\infty$.  Multiplying both sides of the equation (\ref{sol:3}) by $-1$, we turn to Subcase (h2).
 \item {Subcase (h4)} ~ $L^*=-\infty$, $U^*=+\infty$.  Notice that, $L^*=-\infty$ implies that
 either $V^TBV$ has a negative eigenvalue with an
        eigenvector $\widehat{y}$ or there is a vector
        $\widehat{y}\in N(V^TBV)$ such that
        \[{\widehat{y}}^{T}[V^TB(A+\sigma_l^*B)^+(f+\sigma_l^*g)-V^Tg]<0.\]
        Also, $U^*=+\infty$ implies that either $V^TBV$ has
        a positive eigenvalue with an eigenvector
        $\widetilde{y}$ or there is a vector
        $\widetilde{y}\in N(V^TBV)$ such that
        \[{\widetilde{y}}^{T}[V^TB(A+\sigma_l^*B)^+(f+\sigma_l^*g)-V^Tg]>0.\]
        Define
        \begin{eqnarray}
        h_1(\alpha)&:=& {\widehat{y}}^TV^TBV \widehat{y}\alpha^2+2[(f+\sigma_l^*
g)^T(A+\sigma_l^*B)^+BV-g^TV] \widehat{y}\alpha,\\
h_2(\beta)&:=& {\widetilde{y}}^TV^TBV \widetilde{y}\beta^2+2[(f+\sigma_l^*
g)^T(A+\sigma_l^*B)^+BV-g^TV] \widetilde{y} \beta.
        \end{eqnarray}
where $h_1(\alpha)$ is concave quadratic whereas
$h_2(\beta)$ convex quadratic. Since $h_1(0)=h_2(0)=0$, the
ranges of $h_1(\alpha)$ and $h_2(\beta)$, while they are
taken in union, cover the entire $R$. Therefore, if
$\widetilde{\mu}=0$, $y^*=0$ is a solution to (\ref{sol:3}).
If $\widetilde{\mu}<0$, $y^*= \alpha^* \widehat{y}$ with
$h_1(\alpha^*)=\widetilde{\mu}$ is a solution to
(\ref{sol:3}). If $\widetilde{\mu}>0$,
$y^*=\beta^*\widetilde{y}$ with
$h_2(\beta^*)=\widetilde{\mu}$ is the desired solution.
\end{itemize}
\end{itemize}
\end{rem}

\section{(QP1QC) without SDC}\label{SDC_failed}

Simultaneously diagonalizable via congruence (SDC) of a finite
collection of symmetric matrices $A_1, A_2,\ldots,A_m$ is a very
interesting property in optimization due to its tight connection
with the convexity of the cone $\{(\langle A_1 x,x\rangle, \langle
A_2 x,x\rangle,\ldots,\langle A_m x,x\rangle)| x\in R^n\}$. See
\cite{JB-Urruty} for the reference. In \cite{flsx2010}, Feng et al.
concluded that (QP1QC) problems under the SDC condition is either
unbounded below or has an attainable optimal solution whereas those
having no SDC condition are said to be in the {\it hard case}. In
the following, we construct three examples to illustrate why the
hard case is complicate. The examples show that $A$ and $B$ cannot
be simultaneously diagonalizable via congruence, while (QP1QC) could
be unbounded below; could have an unattainable solution; or attain
the optimal value.

\begin{exam}\label{exam:2}
Let
\begin{eqnarray*}
A=\left[\begin{array}{cc}1&0\\0&-1\end{array}\right],~B=\left[\begin{array}{cc}0&1\\1&0\end{array}\right],
~f=g=\left[\begin{array}{c}0\\0\end{array}\right],~\mu=0.
\end{eqnarray*}
If $A$ and $B$ were simultaneously diagonalizable via congruence,
then there would be a nonsingular matrix
$P=\left[\begin{array}{cc}a&b\\c&d\end{array}\right]$ such that
$P^TAP$ and $P^TBP$ are diagonal matrices. That is,
\begin{eqnarray*}
\begin{array}{rl}
P^TAP&=\left[\begin{array}{cc}a&c\\b&d\end{array}\right]\left[\begin{array}{cc}1&0\\0&-1\end{array}\right]
\left[\begin{array}{cc}a&b\\c&d\end{array}\right]
=\left[\begin{array}{cc}a^2-c^2&ab-cd\\ab-cd&b^2-d^2\end{array}\right],
\end{array}
\end{eqnarray*}
and
\begin{eqnarray*}
\begin{array}{rl}
P^TBP&=\left[\begin{array}{cc}a&c\\b&d\end{array}\right]\left[\begin{array}{cc}0&1\\1&0\end{array}\right]
\left[\begin{array}{cc}a&b\\c&d\end{array}\right]
=\left[\begin{array}{cc}2ac&ad+bc\\ad+bc&2bd\end{array}\right]
\end{array}
\end{eqnarray*}
are diagonal matrices. That is, $ab-cd=0$, $ad+bc=0$, and
$ad-bc\neq0$ since $P$ is nonsingular. Since $bc=-ad$, we have
$2ad\neq0$, and hence $a,b,c,$ and $d$ are nonzeros. From $ab-cd=0$,
we have $a=\frac{cd}{b}$; and from $ad+bc=0$, we have
$a=\frac{-bc}{d}$. It implies that $\frac{cd}{b}+\frac{bc}{d}=0$,
which leads to a contradiction that $c(b^2+d^2)=0$ and $b=d=0$. In
other words, $A$ and $B$ cannot be simultaneously diagonalizable via
congruence.

For these $A$ and $B$, (QP1QC) becomes
\begin{eqnarray*}
\begin{array}{rcl}
{\rm (QP1QC)}:~~ \inf &&x_1^2-x_2^2\cr    %
{\rm s.t.} && 2x_1x_2\leq0.
\end{array}
\end{eqnarray*}
We can see that, for any $\sigma\geq0$, $A+\sigma B=\left[\begin{array}{cc}1&\sigma\\\sigma&-1\end{array}\right]$.
Since $(A+\sigma B)_{2,2}=-1$, $A+\sigma B$ cannot be positive semidefinite for any $\sigma\geq0$.
By Theorem \ref{thm:alg:3}, (QP1QC) is unbounded below.
\end{exam}

\begin{exam}\label{exam:3}
Let
\begin{eqnarray*}
A=\left[\begin{array}{cc}0&0\\0&1\end{array}\right],~B=\left[\begin{array}{cc}0&-1\\-1&0\end{array}\right],
~f=g=\left[\begin{array}{c}0\\0\end{array}\right],~\mu=-2.
\end{eqnarray*}
Again, $A$ and $B$ can not be simultaneously diagonalizable via
congruence. Suppose in the contrary there exists a nonsingular
matrix $P=\left[\begin{array}{cc}a&b\\c&d\end{array}\right]$ such
that $P^TAP$ and $P^TBP$ are diagonal matrices:
\begin{eqnarray*}
\begin{array}{rl}
P^TAP&=\left[\begin{array}{cc}a&c\\b&d\end{array}\right]\left[\begin{array}{cc}0&0\\0&1\end{array}\right]
\left[\begin{array}{cc}a&b\\c&d\end{array}\right]
=\left[\begin{array}{cc}c^2&cd\\cd&d^2\end{array}\right],
\end{array}
\end{eqnarray*}
and
\begin{eqnarray*}
\begin{array}{rl}
P^TBP&=\left[\begin{array}{cc}a&c\\b&d\end{array}\right]\left[\begin{array}{cc}0&-1\\-1&0\end{array}\right]
\left[\begin{array}{cc}a&b\\c&d\end{array}\right]
=\left[\begin{array}{cc}-2ac&-ad-bc\\-ad-bc&-2bd.\end{array}\right]
\end{array}
\end{eqnarray*}
Then, we have $cd=0$, $ad+bc=0$, and $ad-bc\neq0$ since $P$ is
nonsingular. If $c=0$, then $ad=0$, which contradicts to
$ad-bc\neq0$. If $d=0$, then $bc=0$, again contradicts to
$ad-bc\neq0$. Hence $A$ and $B$ cannot be SDC.

For these $A$ and $B$, (QP1QC) becomes
\begin{eqnarray*}
\begin{array}{rcl}
{\rm (QP1QC)}:~~ \inf &&x_2^2\cr    %
{\rm s.t.} && x_1x_2\geq1.
\end{array}
\end{eqnarray*}
We can easily see that the optimal solution of (QP1QC) is
unattainable since $x_2$ can be asymptotically approaching 0, but
can not be $0$.
\end{exam}

\begin{exam}\label{exam:4}
Let
\begin{eqnarray*}
A=\left[\begin{array}{cc}0&0\\0&1\end{array}\right],~B=\left[\begin{array}{cc}0&1\\1&0\end{array}\right],
~f=g=\left[\begin{array}{c}0\\0\end{array}\right],~\mu=0.
\end{eqnarray*}
Suppose first that there exists a nonsingular matrix
$P=\left[\begin{array}{cc}a&b\\c&d\end{array}\right]$ such that
$P^TAP$ and $P^TBP$ are diagonal matrices. That is,
\begin{eqnarray*}
\begin{array}{rl}
P^TAP&=\left[\begin{array}{cc}a&c\\b&d\end{array}\right]\left[\begin{array}{cc}0&0\\0&1\end{array}\right]
\left[\begin{array}{cc}a&b\\c&d\end{array}\right]
=\left[\begin{array}{cc}c^2&cd\\cd&d^2\end{array}\right],
\end{array}
\end{eqnarray*}
and
\begin{eqnarray*}
\begin{array}{rl}
P^TBP&=\left[\begin{array}{cc}a&c\\b&d\end{array}\right]\left[\begin{array}{cc}0&1\\1&0\end{array}\right]
\left[\begin{array}{cc}a&b\\c&d\end{array}\right]
=\left[\begin{array}{cc}2ac&ad+bc\\ad+bc&2bd\end{array}\right]
\end{array}
\end{eqnarray*}
are diagonal matrices. It follows that $cd=0$, $ad+bc=0$, and
$ad-bc\neq0$ since $P$ is nonsingular. By $bc=-ad$, we have
$2ad\neq0$, and hence $a,b,c,$ and $d$ are nonzeros, which is
contradicts to $cd=0$. Thus $A$ and $B$ cannot be simultaneously
diagonalizable via congruence. In this example, (QP1QC) becomes
\begin{eqnarray*}
\begin{array}{rcl}
{\rm (QP1QC)}:~~ \inf &&x_2^2\cr    %
{\rm s.t.} && 2x_1x_2\leq0
\end{array}
\end{eqnarray*}
which has the optimal solution set
$\left\{(x,0)|x\in\mathbb{R}\right\}$.
\end{exam}

\section{Conclusions}\label{conclusions}

In this paper, we analyzed the positive semi-definite pencil
$I_{\succeq}(A,B)$ for solving (QP1QC) without any primal Slater
condition, dual Slater condition, or the SDC condition. Given any
(QP1QC) problem, we are now able to check whether it is infeasible,
or unbounded below, or bounded below but unattainable in polynomial
time. If neither of the undesired cases happened, the solution of
(QP1QC) can be obtained via different relatively simple subproblems
in various subcases. In other words, once a given (QP1QC) problem is
properly classified, its solution can be computed readily.
Therefore, we believe that our analysis in this paper has the
potential to become an efficient algorithm for solving (QP1QC) if
carefully implemented.

%
%

\end{document}